\documentclass[a4paper,12pt]{article}
\usepackage[T1]{fontenc}
\usepackage{amsmath, graphicx, color, amsthm, centernot}
\usepackage[font=small,labelfont=bf,width=12.5cm]{caption}
\usepackage{algorithmic, algorithm}

\newtheorem{theorem}{Theorem}
\newtheorem{lemma}[theorem]{Lemma}

\newtheorem{conjecture}{Conjecture}

\newcommand{\set}[1]{\ensuremath{\left\{#1 \right\}}}
\newcommand{\chis}[1]{\chi_{\textrm{st}}'(#1)}

\title{Star edge coloring of some classes of graphs}
\author
{
	\v{L}udmila Bezegov\'{a}\footnotemark[1],\quad
	Borut Lu\v{z}ar\footnotemark[2],\quad
	Martina Mockov\v{c}iakov\'{a}\footnotemark[1], \\
	Roman Sot\'{a}k\thanks{Institute of Mathematics, Faculty of Science, Pavol Jozef \v Saf\'arik University, Ko\v sice, Slovakia. 
			Supported in part by Science and Technology Assistance Agency under the contract No. APVV-0023-10 (R.~Sot\'{a}k),
			Slovak VEGA Grant No.~1/0652/12 (M.~Mockov\v{c}iakova, R. Sot\'ak), and VVGS UPJ\v{S} No.~59/12-13 (M.~Mockov\v{c}iakova).
			All the authors where supported in part by bilateral project SK-SI-0005-10 between Slovakia and Slovenia.
			E-Mails: \texttt{roman.sotak@upjs.sk, \{ludmila.bezegova,martina.mockovciakova\}@student.upjs.sk}},\quad
	Riste \v{S}krekovski\thanks{Faculty of Information Studies, Novo mesto, Slovenia \& Institute of Mathematics, Physics and Mechanics, Ljubljana, Slovenia.  
			Partially supported by ARRS Program P1-0383. E-Mails: \texttt{\{borut.luzar,skrekovski\}@gmail.com}}
}

\begin{document}
\maketitle

\abstract
{
	\textit{A star edge coloring} of a graph is a proper edge coloring without bichromatic paths and cycles of length four. 
	In this paper we establish tight upper bounds for trees and subcubic outerplanar graphs, and derive an upper bound 
	for outerplanar graphs.
}


\section{Introduction}

A \textit{star edge coloring} of a graph $G$
is a proper edge coloring where at least three distinct colors are used on the edges of every path and cycle of length four, 
i.e., there is neither bichromatic path nor cycle of length four. The minimum number of colors for which $G$ admits a star edge coloring
is called the \textit{star chromatic index} and it is denoted by $\chis{G}$.

The star edge coloring was initiated in 2008 by Liu and Deng~\cite{LiuDen08}, motivated by the vertex version (see~\cite{AlbChaKieKunRam04,BuCraMonRasWan09,CheRasWan13,Gru73,KieKunTim09,NesOss03}).
Recently, Dvo\v{r}\'{a}k, Mohar and \v{S}\'{a}mal \cite{DvoMohSam13} determined upper and lower bounds for complete graphs. 
\begin{theorem}[Dvo\v{r}\'{a}k, Mohar, \v{S}\'{a}mal]
	\label{thm:complete}
	The star chromatic index of the complete graph $K_n$ satisfies
	$$
		2n (1 + o(1)) \le \chis{K_n} \le n \frac{2^{2\sqrt{2}(1+o(1))\sqrt{\log{n}}}}{(\log{n}^{1/4})}\,.
	$$
	In particular, for every $\epsilon >0$ there exists a constant $c$ such that $\chis{K_n} \le cn^{1+\epsilon}$ for every $n\ge 1$.
\end{theorem}

They asked what is the true order of magnitude of $\chis{K_n}$, in particular, if $\chis{K_n} = O(n)$.
From Theorem~\ref{thm:complete}, they also derived the following near-linear upper bound in terms of the maximum degree $\Delta$ for general graphs.
\begin{theorem}[Dvo\v{r}\'{a}k, Mohar, \v{S}\'{a}mal]
	Let $G$ be an arbitrary graph of maximum degree $\Delta$. Then
	$$
		\chis{G} \le \chis{K_{\Delta + 1}} \cdot O \bigg( \frac{\log \Delta}{\log \log \Delta} \bigg)^2\,,
	$$
	and therefore $\chis{G} \le \Delta \cdot 2^{O(1)\sqrt{\log \Delta}}$.
\end{theorem}
In addition, Dvo\v{r}ak et al. considered cubic graphs and showed that the star chromatic index lies between $4$ and $7$.
\begin{theorem}[Dvo\v{r}\'{a}k, Mohar, \v{S}\'{a}mal]~
	\begin{itemize}
		\item[$(a)$] If $G$ is a subcubic graph, then $\chis{G} \le 7$.
		\item[$(b)$] If $G$ is a simple cubic graph, then $\chis{G} \ge 4$, and the equality holds
			if and only if $G$ covers the graph of the $3$-cube.
	\end{itemize}
\end{theorem}
A graph $G$ \textit{covers} a graph $H$ if there is a locally bijective graph homomorphism from $G$ to $H$. While there
exist cubic graphs with the star chromatic index equal to $6$, e.g., $K_{3,3}$ or Heawood graph, no example of a subcubic graph that would require $7$ colors
is known. Thus,  Dvo\v{r}ak et al. proposed the following conjecture.
\begin{conjecture}[Dvo\v{r}\'{a}k, Mohar, \v{S}\'{a}mal]
	If $G$ is a subcubic graph, then $\chis{G}~\le~6$.
\end{conjecture}

In this paper we determine a tight star chromatic index of subcubic outerplanar graphs and trees, and an upper bound for star chromatic index of outerplanar graphs. 

Throughout the paper, we mostly use the terminology used in~\cite{BonMur08}.
In a graph with colored edges, we say that a color $c$ \textit{appears} at vertex $v$, whenever there is an edge
incident with $v$ and colored by $c$. In a rooted tree, the root is considered to be at \textit{level $0$}, its neighbors at \textit{level $1$} ($1$-\textit{level vertices}), etc. 
An edge is of \textit{level $i$} ($i$-\textit{level edge}) if its two endvertices are at level $i$ and $i+1$. A path of four edges is referred to as a $4$-path.

\section{Trees}

A tight upper bound for acyclic graphs is presented in this section. For example, a star chromatic index of every tree with a $\Delta$-vertex whose 
all neighbors are $\Delta$-vertices achieves the upper bound.
\begin{theorem}
	\label{thm:trees}
	Let $T$ be a tree with maximum degree $\Delta$. Then
	$$
		\chis{T} \le \bigg \lfloor \frac{3}{2} \Delta \bigg \rfloor\,.
	$$
	Moreover, the bound is tight.
\end{theorem}

\begin{proof}	
	First we prove the tightness of the bound. Let $T$ be a rooted tree with a root $v$ of degree $\Delta$ and let each of its neighbors $v_i$ have 
	$\Delta-1$ leafs. Thus, $T$ has $\Delta (\Delta-1)$ $1$-level edges. In an optimal coloring with $\Delta + k$ colors suppose that each edge $vv_i$ is 
	colored by $i$. Note that for any distinct neighbors $v_i$, $v_j$ of $v$ we have that if the color $j$ appears at $v_i$ the color $i$ does not
	appear at $v_j$, and vice versa. Thus, colors $1,\ldots, \Delta$ may appear together at most $\Delta(\Delta-1)/2$ times at $1$-level edges. 
	Moreover, each of the colors $\Delta+1,\ldots,\Delta + k$ may appear at each $v_i$, so together they appear $k\Delta$ times at $1$-level edges. Hence, it must hold 	
	$$
		\Delta(\Delta-1)/2 + k\,\Delta \ge \Delta(\Delta-1)\,,
	$$ 
	which implies that $k \ge (\Delta-1)/2$ and gives the tightness of our bound.
    	
	\bigskip
	Now, we present a construction of a star edge coloring $\varphi$ of $T$ using colors from a set $C$, where 
	$|C| = \big\lfloor\frac{3}{2}\Delta \big\rfloor$.
	First, we select an arbitrary vertex $v$ of $T$ to be the root of $T$. Moreover, to every vertex of degree at least $2$ in $T$
	we add some edges such that its degree increases to $\Delta$. In particular, every vertex in $T$ has degree $\Delta$ or $1$ after
	this modification. For every vertex $u$ of $T$ we denote its neighbors by $n_0(u),\dots,n_{\Delta-1}(u)$; here we always choose its eventual predecessor 
	regarding the root $v$ to be $n_0(u)$. By $t(u)$, for $u \neq v$, we denote an integer $i$ such that $u = n_i(n_0(u))$, and by $C(u)$ we denote
	the current set of colors of the edges incident to $u$. 
	
	We obtain $\varphi$ in the following way. First we color the edges incident to the root $v$ such that 
	$\varphi(vn_i(v)) = i + 1$, for $i = 0,1,\dots,\Delta-1$. Then, we continue by coloring the edges incident to the remaining vertices,
	in particular, we first color the edges incident to the vertices of level 1, then the edges incident to the vertices of level 2 etc. 
	
	Let $u$ be a $k$-level vertex whose incident $k$-level edges are not colored yet. Let $C^\prime(n_0(u)) = C\setminus C(n_0(u))$ be the set
	of colors that are not used for the edges incident to $n_0(u)$. 
	Since $|C^\prime(n_0(u))|=\big\lfloor\frac{\Delta}{2} \big\rfloor$, for the edges $un_{\Delta-i}(u)$ for $i=1,\dots,\big\lfloor\frac{\Delta}{2} \big\rfloor$ 
	we use distinct colors from the set $C^\prime(n_0(u))$.
	
	Next, we color the remaining edges incident to $u$. There are $\big\lceil \frac{\Delta}{2} \big\rceil - 1$ of such edges. 
	We color the edges one by one for $i = 1,\dots, \big\lceil \frac{\Delta}{2} \big\rceil -1$ $(*)$ with colors from $C(n_0(u))$ as follows.
	Let $j = (t(u) + i) \bmod{\Delta}$ and set
	\begin{eqnarray}
		\label{eq:2stars}
		\varphi(un_i(u)) = \varphi(n_0(u)n_{j}(n_0(u)))\,.
	\end{eqnarray}	
	Notice that after the above procedure is performed for each vertex $u$ all the edges of $T$ are colored. 
	
	It remains to verify that $\varphi$ is indeed a star edge coloring of $T$. It is easy to see that the edges incident to the same vertex 
	receive distinct colors, so we only need to show that there is no bichromatic $4$-path in $T$.
	Suppose, for a contradiction, that there is such a path $P = wxyzt$ with $\varphi(xw) = \varphi(yz)$ and $\varphi(xy) = \varphi(tz)$. 
	Assume, without loss of generality, that $d(v,w) > d(v,z)$. Note also that $y = n_0(x)$, $w = n_i(x)$, for some $i > 0$, and $z = n_j(y)$, for some $j \ge 0$.
	Since 
	\begin{eqnarray*}
		\varphi(xw) &=& \varphi(yz) \\
		\varphi(xn_i(x)) &=& \varphi(n_0(x)n_j(n_0(x)))\,,
	\end{eqnarray*}
	by~\eqref{eq:2stars} we have that $j = (t(x) + i) \bmod{\Delta}$. Hence, $z = n_{t(x) + i}(y)$ and by $(*)$ we have
	\begin{eqnarray}
		\label{eq:a}
		1\le i \le \bigg\lceil \frac{\Delta}{2} \bigg\rceil -1.
	\end{eqnarray}
	Now, we consider two cases regarding the value of $j$. First, suppose $j = (t(x) + i) \bmod{\Delta} = 0$ (it means that $i + t(x) = \Delta$).
	Since 
	\begin{eqnarray*}
		\varphi(z t) &=& \varphi(y x) \\
		\varphi(n_0(y) n_\ell(n_0(y))) &=& \varphi(y n_{t(x)}(y))\,,
	\end{eqnarray*}
	as above, by~\eqref{eq:2stars}, we have that $\ell = (t(y) + t(x)) \bmod{\Delta}$ and by $(*)$
	\begin{eqnarray}
		\label{eq:b}
		1\le t(x) \le \bigg\lceil \frac{\Delta}{2} \bigg\rceil -1.
	\end{eqnarray}
	By summing~\eqref{eq:a} and~\eqref{eq:b} and using $i + t(x) = \Delta$, we obtain
	
	\begin{eqnarray*}
		\Delta = i + t(x) \le 2 \bigg \lceil \frac{\Delta}{2} \bigg \rceil-2 < \Delta,
	\end{eqnarray*}
	a contradiction. Second, let $j > 0$. Then, $y = n_0(z)$ and
	\begin{eqnarray*}
		\varphi(y x) &=& \varphi(z t) \\
		\varphi(n_0(z) n_{t(x)}(n_0(z))) &=& \varphi(z n_{\ell}(z))\,.
	\end{eqnarray*}
	By~\eqref{eq:2stars}, $t(x) = (t(z) + \ell) \bmod{\Delta}$ and 
	\begin{eqnarray}
		\label{eq:c}
		1\le \ell \le \bigg \lceil \frac{\Delta}{2} \bigg \rceil-1\,.
	\end{eqnarray}
	Moreover, since $j = t(x) + i$ and $t(z) = j$, we have
	$$
		t(x) = (t(z) + \ell) \bmod{\Delta}  = (t(x) + i + \ell) \bmod{\Delta}\,.
	$$
	So, $i + \ell = \Delta$ and, similarly as in the previous case, by summing~\eqref{eq:a} and~\eqref{eq:c} we obtain
	\begin{eqnarray*}
			\Delta =  i + \ell \le 2\bigg \lceil \frac{\Delta}{2} \bigg \rceil < \Delta,
	\end{eqnarray*}
	a contradiction that establishes the theorem.
\end{proof}

\section{Outerplanar graphs}

In this section we consider outerplanar graphs. First, we use the bound from Theorem~\ref{thm:trees} to derive an upper bound for outerplanar graphs.
\begin{theorem}
	\label{thm:outer}
	Let $G$ be an outerplanar graph with maximum degree $\Delta$. Then,
	$$
		\chis{G} \le \bigg \lfloor \frac{3}{2} \Delta \bigg \rfloor + 12\,.
	$$
\end{theorem}

\begin{proof}
	Let $G$ be an outerplanar graph and $v$ a vertex of $G$. Let $T$ be a rooted spanning tree of $G$ with the root $v$ such that for any vertex 
	$u \in V(G)$ is $d_G(v,u) = d_T(v,u)$. By Theorem~\ref{thm:trees}, there exists a star edge coloring $\varphi$ of $T$ with at most $\big \lfloor \frac{3}{2} \Delta \big \rfloor$ colors.

	Now, we consider the edges of $G$ which are not contained in $T$. Let $uw \in E(G)\setminus E(T)$ be an edge such that $d(v,u) = k$, 
	$d(v,w) = \ell$, and $k \le \ell$. By the definition of $T$, we have that either $k = \ell$ or $k = \ell - 1$. In the former case, we refer to
	$uw$ as a \textit{horizontal edge} (or a \textit{$k$-horizontal edge}), and in the latter case, we call it a \textit{diagonal edge} (or a \textit{$k$-diagonal edge})
	(see Fig.~\ref{fig:outer} for an example).
	\begin{figure}[ht]
		$$
			\includegraphics{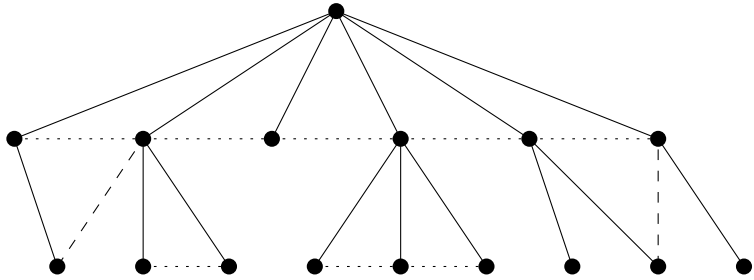}
		$$
		\caption{An outerplanar graph with the edges of the spanning tree depicted solid, the horizontal edges dotted and the diagonal edges dashed.}
		\label{fig:outer}
	\end{figure}

	Since outerplanar graphs do not contain $K_{4}$-minors, we have that the horizontal edges induce a linear forest in $G$, i.e., a graph in which
	every component is a path. It is easy to see that $3$ colors suffice for a star coloring of a path. We use three additional colors
	for all the $k$-horizontal edges when $k$ is odd, and another three colors for the $k$-horizontal edges when $k$ is even, six additional colors altogether.
	Obviously, no bichromatic $4$-path is introduced in $G$.

	It remains to color the diagonal edges. Let $L_k$ be the bipartite graph whose vertices are the $k$-level and $(k+1)$-level
	vertices of $G$, and two vertices are adjacent if they are adjacent in $T$ (\textit{red edges}), or they are connected by a $k$-diagonal edge (\textit{blue edges}). 
	Since $G$ is $K_{2,3}$-minor free, $L_k$ is acyclic. 

	Every $k$-level vertex is incident to at most two blue edges in $L_k$, and every 
	$(k+1)$-level vertex is incident to at most one blue edge in $L_k$. If a $k$-level vertex $w$ is incident to two blue
	edges $wx$ and $wy$, none of its successors in $T$ is incident to a blue edge in $L_k$, and the predecessors of $x$ and $y$ in $T$ are incident to at most one
	blue edge in $L_k$. 

	This properties imply that the $k$-diagonal edges can be colored with two colors such that no bichromatic path is introduced in $G$.
	For instance, we first color the edges incident to eventual $k$-level vertices that are incident to two $k$-diagonal edges, and then the remaining edges.
	We use $2$ distinct colors for the $k$-diagonal edges when $k \equiv 0 \pmod{3}$, another two when $k \equiv 1 \pmod{3}$, and two when $k \equiv 2 \pmod{3}$, so $6$ altogether.

	Hence, we obtain a star edge coloring of $G$ with at most $\lfloor \frac{3}{2} \Delta \rfloor$ + $12$ colors.
\end{proof}

In a special case when a maximum degree of an outerplanar graph is $3$, we prove that
five colors suffice for a star edge coloring. Moreover, the bound is tight. Consider a $5$-cycle $v_1v_2v_3v_4v_5$ with one pendent edge at every vertex 
(see Figure~\ref{fig:outersub}). 
\begin{figure}[ht]
	$$
		\includegraphics{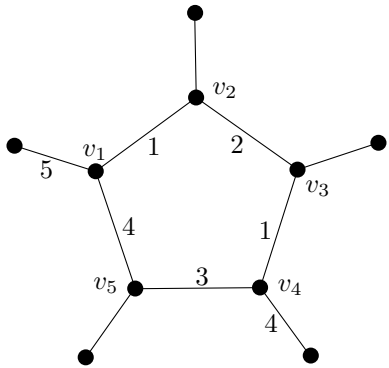}
	$$
	\caption{A subcubic outerplanar graph with the star chromatic index $5$.}
	\label{fig:outersub}
\end{figure}
In case that only four colors suffice, all four should be used on the
edges of the cycle and only one could be used twice, so the coloring of the cycle is unique up to the permutation of colors. 
Consider the coloring of the cycle in Figure~\ref{fig:outersub}. Only the color $4$ may be used for the pendent edge incident to $v_4$.
But, in this case, all the colors are forbidden for the pendent edge incident to $v_1$, which means that we need at least
$5$ colors for a star edge coloring of the graph. In particular, we may color all the pendent edges of the graph by color $5$
and obtain a star $5$-edge coloring.

Before we prove the theorem, we present a technical lemma. A \textit{cactus graph} is a (multi)graph in which 
every block is a cycle or an edge. We say that a cactus graph is a \textit{light} cactus graph if the following properties hold:
\begin{itemize}
	\item all the vertices have degree $2$ or $4$;
	\item only the vertices of degree $4$ are cutvertices;
	\item on every cycle there are at most two vertices of degree $2$. Moreover, if there are two vertices of degree $2$ 
			on a cycle, then they are adjacent.
\end{itemize}
  
\begin{lemma}
	\label{lem:cactus}
	Every outerplanar embedding of a light cactus graph admits a proper $4$-edge coloring such that no 
	bichromatic $4$-path exists on the boundary of the outer face.
\end{lemma}

\begin{proof}
	The proof is by induction on the number $n_c$ of cutvertices of a light cactus graph $L$. The base case 
	with $n_c = 0$ is trivial, since $L$ has at most two edges. In case when $n_c = 1$, $L$ is isomorphic to one of the three graphs
	depicted in Fig.~\ref{fig:subcubic1}, together with the colorings of the edges.
	\begin{figure}[htb]
		$$
			\includegraphics{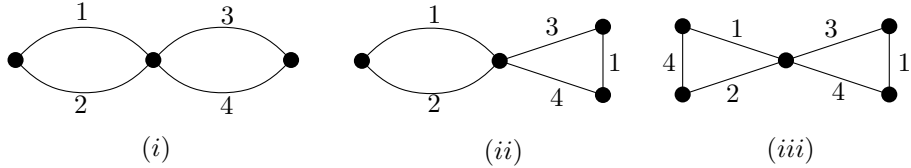}
		$$
		\caption{The three light cactus graphs with one cutvertex.}
		\label{fig:subcubic1}
	\end{figure} 

	Now, let $L$ contain $k$ cutvertices, $k \ge 2$. Then, there is a block $M$ of $L$ that contains at least two cutvertices of $L$. 
	Let $v_1, v_2,\dots,v_\ell$, $2 \le \ell \le k$, be the cutvertices of $L$ in $M$. Every $v_i$ is incident to two edges $t_i^1$ and $t_i^2$ of $L$ 
	such that $t_i^1,t_i^2 \notin E(M)$, we call such edges \textit{tentacles}. Note that these two edges belong to a same block.

	Let $L_i$ be the component of $L \setminus E(M)$ containing $v_i$ (and so $t_i^1$ and $t_i^2$). Let $H_i$ be the graph obtained from $L_i$ by connecting 
	the vertex $v_i$ with a new vertex $u_i$ with two edges $e_i$ and $f_i$. By induction hypothesis, every $H_i$ admits a coloring with at most $4$ colors. 

	Next, let $L_M$ be a subgraph of $L$ consisting of the cycle $M$ and all its tentacles.
	We color the edges of $L_M$ as follows. If $M$ is an even cycle, we color the edges of $M$ by the colors $1$ and $2$, 
	and the tentacles by colors $3$ and $4$ in such a way that the tentacles of color $3$ are always on the same side, i.e., contraction of the edges
	of $M$ results in a star with the edges alternately colored by $3$ and $4$ (see Fig.~\ref{fig:tentacles}).
	\begin{figure}[htb]
		$$
			\includegraphics{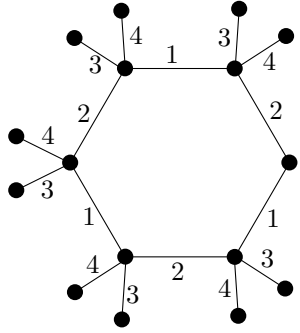}
		$$
		\caption{An example of an edge coloring of a graph $L_M$, when $M$ is an even cycle.}
		\label{fig:tentacles}
	\end{figure} 

	In case when $M$ is an odd cycle, it contains a pair of adjacent cutvertices $u$ and $v$ of $L$. Without loss of generality, we choose $u$ to be 
	incident to an eventual vertex of degree $2$ of $M$. 
	We color the edge $uv$ with color $3$, and the remaining edges of $M$ alternately by the colors $1$ 
	and $2$ starting at the edge incident to $u$. Next, we color the tentacles by colors $3$ and $4$
	similarly as in the previous case such that the tentacles of color $3$ are always on the same side, in particular the edge of color $4$ incident
	to $u$ is consecutive to the edge of color $1$ incident to $u$ on the boundary of the outerface of $L_M$. In the end, we recolor the tentacles incident to
	the vertices $u$ and $v$ as follows. 
	
	In case when $M$ is a cycle of length $3$ with a vertex of degree $2$, we recolor the tentacle of color $3$ incident 
	to $u$ by color $2$, and the tentacles of color $4$ and $3$ incident to $v$ by colors $1$ and $4$, respectively
	(see the left graph in Fig.~\ref{fig:tentaclesodd}). Otherwise, if $M$ is not a cycle of length $3$ with a vertex of degree $2$, 
	we recolor the tentacle of color $3$ incident to $u$ by color $2$, and the 
	tentacle of color $3$ incident to $v$ by color $1$ (see the right graph in Fig.~\ref{fig:tentaclesodd} for an example).
	\begin{figure}[htb]
		$$
			\includegraphics{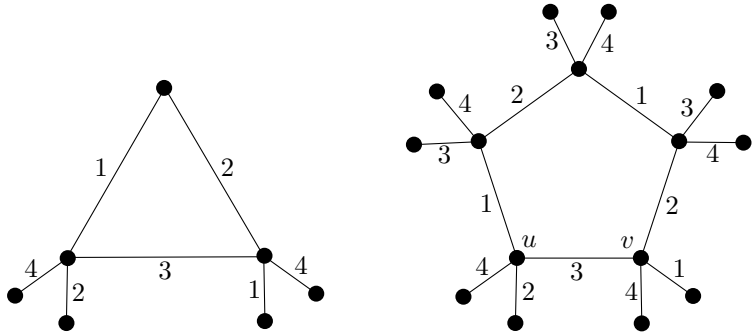}
		$$
		\caption{Examples of an edge coloring of a graph $L_M$, when $M$ is an odd cycle.}
		\label{fig:tentaclesodd}
	\end{figure} 
	Observe that every four consecutive edges along the boundary of the outer face in $L_M$ receive at least three distinct colors. 
	Moreover, every three consecutive edges along the boundary of the outer face of $L_M$ starting at a tentacle receive three distinct colors.

	Finally, we permute the colors in each component $L_i$ to match the coloring of $L_M$ (we identify the edges $e_i$ and $f_i$ with 
	the tentacles of $M$ incident to $v_i$). It is easy to see that we obtain an edge coloring without bichromatic $4$-path along the boundary 
	of the outer face of the graph $L$. Namely, let $P$ be a $4$-path along the outer face in $L$. If at least three edges of $P$ are in $L_M$, 
	then $P$ is not bichromatic. On the other hand, if $L_M$ contains at most two edges of $P$, then all the edges of $P$ belong to some component $L_i$.
	Hence, we have shown that every outerplanar embedding of a light cactus graph $L$ admits a proper $4$-edge coloring such that no 
	bichromatic $4$-path exists on the boundary of the outer face, and so establish the lemma.
\end{proof}

\begin{theorem}
	\label{thm:outersub}
	Let $G$ be a subcubic outerplanar graph. Then,
	$$
		\chis{G} \le 5\,.
	$$
\end{theorem}

\begin{proof}
	Let $G$ be a minimal counterexample to the theorem, i.e., a subcubic outerplanar graph with the minimum number of edges that
	does not admit a star edge coloring with at most $5$ colors. We also assume that $G$ is embedded in the plane.
	We first discuss a structural property of $G$ and then show that there exists a star $5$-edge coloring in $G$, hence obtaining a contradiction.
	
	Suppose that $G$ contains a bridge $uv$, where $u$ and $v$ have degree at least $2$ (we call the edge $uv$ a \textit{nontrivial bridge}). 
	Let $G_u$ and $G_v$ be the two components of $G - uv$ containing the vertex $u$ and $v$, respectively. 
	Let $H_u = G[V(G_u) \cup \set{v}]$, i.e., a subgraph of $G$ induced by the vertices in $V(G_u) \cup \set{v}$, and similarly, let $H_v = G[V(G_v) \cup \set{u}]$. 
	By the minimality of $G$, there exist star $5$-edge colorings of $H_u$ and $H_v$. We may assume (eventually after a permutation of the
	colors), that $uv$ is given the color $5$ in both graphs $H_u$ and $H_v$ and all the edges incident to the edge $uv$ are colored by distinct colors. 
	It is easy to see that the colorings of $H_u$ and $H_v$ induce a star $5$-edge coloring of $G$, a contradiction.
	
	Hence, we may assume that $G$ has no nontrivial bridge, so $G$ is comprised of an outer cycle $C = v_1v_2\dots v_k$, some chords of $C$ and some pendent edges. 
	Now, we construct a star $5$-edge coloring of $G$. We will show that there exists a star edge coloring of the edges of $C$ by the 
	four colors in such a way that every chord of $C$ is incident to the four distinctly colored edges.
	That means that after we color the chords and pendent edges by the fifth color a star $5$-edge coloring of $G$ is obtained.
	
	First, we remove the pendent edges and then maximize the number of chords of $C$ by joining pairs of nonadjacent $2$-vertices incident to common 
	inner faces. We obtain a $2$-connected subcubic outerplanar graph $G'$ which is either isomorphic to $C_3$ or every inner face $f$ of $G'$ is incident 
	to at most two $2$-vertices. Moreover, if $f$ is incident to two $2$-vertices, they are adjacent. 
	
	Now, let $G''$ denote the graph obtained from 
	$G'$ by contracting the chords. Observe that $G''$ is an outerplanar multigraph with vertices of degree $2$ or $4$, where the $4$-vertices are the 
	cutvertices obtained by chord contractions, and the $2$-vertices of $G''$ are the $2$-vertices of $G'$. 
	In fact, $G''$ is a light cactus graph (see Fig.~\ref{fig:lightcactus} for an example).
	\begin{figure}[htb]
			$$
				\includegraphics{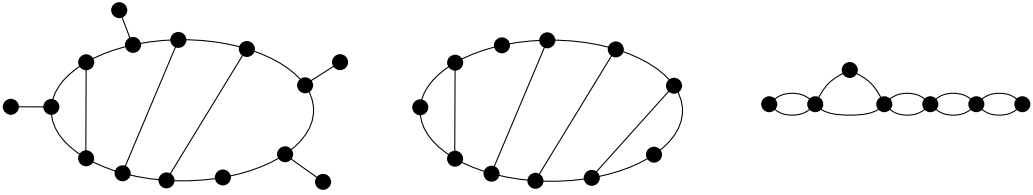} 
			$$
			\caption{An example of the graphs $G$, $G'$ and $G''$.}
			\label{fig:lightcactus}
	\end{figure}	

	It remains to find a star coloring of $G$. By Lemma~\ref{lem:cactus}, we can color the edges of $G''$ with four colors such that no 
	bichromatic $4$-path exists on the boundary of the outer face. This coloring induces a partial edge coloring of $G$, where only the chords of $C$ and
	the pendent edges remain noncolored. We color them by color $5$. Since all the edges incident to any edge of color $5$ receive distinct colors,
	we obtain a star $5$-edge coloring of $G$. This completes the proof.
\end{proof}

In Theorem~\ref{thm:outer}, we proved that $\big \lfloor \frac{3}{2} \Delta \big \rfloor + 12$ colors suffice for a star edge coloring
of every outerplanar graph. Let us mention that the constant $12$ can be decreased to $9$ using more involved analysis.
However, we prefer to present a short proof, since the precise bound is probably considerably lower. In particular, we believe that the following conjecture holds.
\begin{conjecture}
	Let $G$ be an outerplanar graph with maximum degree $\Delta\ge 3$. Then,
	$$
		\chis{G} \le \bigg \lfloor \frac{3}{2} \Delta \bigg \rfloor +1\,.
	$$	
\end{conjecture}

For graphs with maximum degree $\Delta=2$, i.e. for paths and cycles, there exist star edge colorings with at most $3$ colors except 
for $C_5$ which requires $4$ colors. In case of subcubic outerplanar graphs the conjecture is confirmed by Theorem~\ref{thm:outersub}.

\bibliographystyle{acm}
{\small
	\bibliography{MainBase}

\begin{thebibliography}{1}

\bibitem{AlbChaKieKunRam04}
{\sc Albertson, M.~O., Chappell, G.~G., Kiersted, H.~A., K\"{u}nden, A., and
  Ramamurthi, R.}
\newblock Coloring with no $2$-colored ${P}_4$'s.
\newblock {\em Electron. J. Combin. 1\/} (2004), \#R26.

\bibitem{BonMur08}
{\sc Bondy, J.~A., and Murty, U. S.~R.}
\newblock {\em Graph theory}, vol.~244 of {\em Graduate Texts in Mathematics}.
\newblock Springer, New York, 2008.

\bibitem{BuCraMonRasWan09}
{\sc Bu, Y., Cranston, N.~W., Montassier, M., Raspaud, A., and Wang, W.}
\newblock Star-coloring of sparse graphs.
\newblock {\em J. Graph Theory 62\/} (2009), 201--219.

\bibitem{CheRasWan13}
{\sc Chen, M., Raspaud, A., and Wang, W.}
\newblock 6-star-coloring of subcubic graphs.
\newblock {\em J. Graph Theory 72}, 2 (2013), 128--145.

\bibitem{DvoMohSam13}
{\sc Dvo\v{r}\'{a}k, Z., Mohar, B., and \v{S}\'{a}mal, R.}
\newblock Star chromatic index.
\newblock {\em J. Graph Theory 72\/} (2013), 313--326.

\bibitem{Gru73}
{\sc Gr{\" u}nbaum, B.}
\newblock Acyclic coloring of planar graphs.
\newblock {\em Israel J. Math. 14\/} (1973), 390--412.

\bibitem{KieKunTim09}
{\sc Kierstead, H.~A., K\"{u}ndgen, A., and Timmons, C.}
\newblock Star coloring bipartite planar graphs.
\newblock {\em J. Graph Theory 60\/} (2009), 1--10.

\bibitem{LiuDen08}
{\sc Liu, X.-S., and Deng, K.}
\newblock An upper bound on the star chromatic index of graphs with $\delta \ge
  7$.
\newblock {\em J. Lanzhou Univ. (Nat. Sci.) 44\/} (2008), 94--95.

\bibitem{NesOss03}
{\sc Ne\v{s}et\v{r}il, J., and de~Mendez, P.~O.}
\newblock Colorings and homomorphisms of minor closed classes.
\newblock {\em Algorithms Combin. 25\/} (2003), 651--664.

\end{thebibliography}
}

\end{document}